\documentclass[a4paper]{amsart}
\usepackage{amsmath,amsthm,amssymb,hyperref,mathrsfs}
\usepackage{aliascnt}
\usepackage{lmodern}
\usepackage[T1]{fontenc}

\usepackage[textsize=footnotesize]{todonotes}

\usepackage{xcolor}
\definecolor{dblue}{rgb}{0,0,0.70}
\hypersetup{
	unicode=true,
	colorlinks=true,
	citecolor=dblue,
	linkcolor=dblue,
	anchorcolor=dblue
}

\makeatletter
\expandafter\g@addto@macro\csname th@plain\endcsname{%
		\thm@notefont{\bfseries}
	}%
\expandafter\g@addto@macro\csname th@remark\endcsname{%
		\thm@headfont{\bfseries}
	}%
\makeatother

\newtheorem{theorem}
{Theorem}[section]	
\newtheorem*{theorem*}{Theorem}

\newaliascnt{lemma}{theorem}
\newtheorem{lemma}[lemma]{Lemma}
\newtheorem{claim}[theorem]{Claim}
\aliascntresetthe{lemma}
\newtheorem*{lemma*}{Lemma}

\newaliascnt{proposition}{theorem}
\newtheorem{proposition}[proposition]{Proposition}
\aliascntresetthe{proposition}

\newaliascnt{corollary}{theorem}

\aliascntresetthe{corollary}

\theoremstyle{remark}

\newaliascnt{remark}{theorem}
\newtheorem{remark}[remark]{Remark}
\aliascntresetthe{remark}
\newaliascnt{question}{theorem}
\newtheorem{question}[question]{Question}
\aliascntresetthe{question}

\newtheorem*{question*}{Question}

\newaliascnt{definition}{theorem}
\newtheorem{definition}[definition]{Definition}
\aliascntresetthe{definition}

\newaliascnt{example}{theorem}

\aliascntresetthe{example}

\renewcommand{\restriction}{\mathbin\upharpoonright}

\newcommand{\axiom}[1]{\mathsf{#1}} 
\newcommand{\ZFC}{\axiom{ZFC}}
\newcommand{\AC}{\axiom{AC}}
\newcommand{\CH}{\axiom{CH}}
\newcommand{\AD}{\axiom{AD}}
\newcommand{\DC}{\axiom{DC}}
\newcommand{\ZF}{\axiom{ZF}}

\newcommand{\Ord}{\mathrm{Ord}}

\newcommand{\PFA}{\axiom{PFA}}

\newcommand{\A}{\axiom{A}}
\newcommand{\MRP}{\axiom{MRP}}
\newcommand{\UltimateL}{\textrm{Ultimate-}L}

\DeclareMathOperator{\CFA}{CFA}
\DeclareMathOperator{\cf}{cf}
\DeclareMathOperator{\dom}{dom}

\DeclareMathOperator{\rank}{rank}

\DeclareMathOperator{\Col}{Col}
\DeclareMathOperator{\Add}{Add}

\DeclareMathOperator{\tcl}{tcl}

\newcommand{\forces}{\mathrel{\Vdash}}

\newcommand\PP{\mathbb{P}}
\newcommand{\power}{\mathcal{P}}

\newcommand{\QQ}{\mathbb{Q}}
\newcommand{\RR}{\mathbb{R}}

\newcommand{\cD}{\mathcal D}

\newcommand{\cC}{\mathcal C}
\newcommand{\cQ}{\mathcal Q}
\newcommand{\cL}{\mathcal L}

\newcommand{\tup}[1]{\langle#1\rangle}

\author{David Asper\'o}
\author{Asaf Karagila}
\thanks{The first author acknowledges support of EPSRC Grant EP/N032160/1. The second author was partially supported by the Royal Society Newton International Fellowship, grant no.~NF170989.}
\thanks{This research was initiated when the authors were visiting fellows at the Isaac Newton Institute for Mathematical Sciences during the programme `Mathematical, Foundational and Computational Aspects of the Higher Infinite' (HIF) funded by EPSRC grant EP/K032208/1.}
\email[David Asper\'o]{d.aspero@uea.ac.uk}
\email[Asaf Karagila]{karagila@math.huji.ac.il}
\address{School of Mathematics,
University of East Anglia.
Norwich, NR4~7TJ, UK
}
\date{\today}
\subjclass[2010]{Primary 03E25; Secondary 03E55, 03E35, 03E57, 03A05}
\keywords{Axiom of Choice, forcing axioms, proper forcing, PFA, generic absoluteness, the Chang model}

\title[Dependent Choice, Properness, and Generic Absoluteness]{Dependent Choice, Properness,\\ and Generic Absoluteness}

\begin{document}
\begin{abstract}
We show that Dependent Choice is a sufficient choice principle for developing the basic theory of proper forcing, and for deriving generic absoluteness for the Chang model in the presence of large cardinals, even with respect to $\DC$-preserving symmetric submodels of forcing extensions. Hence, $\ZF+\DC$ not only provides the right framework for developing classical analysis, but is also the right base theory over which to safeguard truth in analysis from the independence phenomenon in the presence of large cardinals. We also investigate some basic consequences of the Proper Forcing Axiom in $\ZF$, and formulate a natural question about the generic absoluteness of the Proper Forcing Axiom in $\ZF+\DC$ and $\ZFC$. Our results confirm $\ZF + \DC$ as a natural foundation for a significant portion of ``classical mathematics'' and provide support to the idea of this theory being also a natural foundation for a large part of set theory. 
\end{abstract}
\maketitle
\section{Introduction}
The Axiom of Choice is a staple of modern mathematics which requires little introduction nowadays. However, as the reader is probably aware, the Axiom of Choice also has some controversial implications, especially in analysis (e.g., the existence of sets of reals which are not Lebesgue-measurable or without the Baire property). The controversial aspect of such results, and of the Axiom of Choice itself, ultimately stems from the non-constructive\footnote{In most senses of `constructive' as used in mathematics, as well as in the colloquial sense of the word.} nature of this axiom. Appealing to the Axiom of Choice means that we are unable to properly construct and define our objects, but merely prove their existence. To some people this approach is lacking substance; they would argue that the objects we construct in mathematics should be explicit. At any rate, effective proofs, or proofs which explicitly exhibit or construct some object, are perceived almost universally as being more satisfactory than mere existence proofs. 

By weakening the Axiom of Choice to one of its countable variants we obtain the \textit{Principle of Dependent Choice} ($\DC$), namely the statement that every partial order without maximal points has an infinite chain. This axiom is regarded  by most mathematicians as innocuous, as it does not give rise to the construction of pathological objects in analysis and elsewhere in mathematics.\footnote{For example, Solovay produced a model of $\ZF+\DC$ in which all sets of reals are Lebesgue measurable and have the Baire property (\cite{Solovay:1970}).} By its very formulation $\DC$ is simply telling us that we can construct things recursively in $\omega$ steps just by showing that the construction may proceed through each successive step. Even if the object promised to exist is still non-constructive, this level of non-constructiveness is felt to be a minimum that we should be happy to tolerate if we are to develop any reasonable mathematics of infinite objects.\footnote{Notably, if we are to develop real analysis in any reasonable way.} In fact, and precisely because performing recursive constructions over the natural numbers is something we do all the time in many areas of mathematics, $\DC$ promises to be a particularly successful axiom in those areas. This is actually true for classical analysis, most of which can be carried out in the theory $\ZF+\DC$.

The above considerations on the naturalness of $\DC$ are meant as a motivational preamble to the work in the present paper. This is a paper in two parts. The first part studies to what extent is $\ZF+\DC$ an adequate base to develop the theory of proper forcing, whereas the second part deals with the problem of deriving generic absoluteness for (some substantial class of) forcing notions in the theory $\ZF+\DC$ supplemented with sufficiently strong large cardinal axioms. These two lines of research may look rather disconnected, but they share a goal, namely to see whether $\ZF+\DC$, which we have hinted at being a successful foundation for large portions of classical mathematics, is also a reasonable vantage point from which to look into more purely set-theoretic areas. 

Our motivation for studying properness in $\ZF+\DC$ originated in the earlier realization, by the second author, that $\DC$ enables one to prove, and in fact is equivalent to, the existence of countable elementary submodels of $V_\alpha$, for all infinite ordinals $\alpha$;\footnote{This was most likely observed by many others in the past, as George Boolos made the observation in \cite{BoolosJeffery} that $\DC$ is equivalent to the L\"owenheim--Skolem theorem.} in particular, $\ZF+\DC$ provides a non-vacuous playground in which to study properness. On the other hand, the goal of our study of generic absoluteness over the base theory $\ZF+\DC$ was to elucidate whether or not the use of the Axiom of Choice in its full power is an overkill for proving generic absoluteness for $L(\RR)$, or even for the Chang model, from large cardinals. This question is motivated by the role played by $\ZF+\DC$ as a valid base theory for real analysis and the fact that most of real analysis lives in $L(\RR)$.\footnote{Another, less obvious, point of connection between our two lines of research is given by the fact that the first proof of generic absoluteness from large cardinals in $\ZFC$ went through the construction, assuming the existence of a supercompact cardinal, of a semiproper forcing iteration, with semiproperness being a natural outgrowth of the notion of properness. At the beginning of our research we focused on the study of properness and observed the efficiency of $\ZF+\DC$ in this context, to the point of allowing one to force the Proper Forcing Axiom from the existence of a supercompact cardinal. But then, given the above connection, we set to explore whether $\ZF+\DC$ could also have something interesting to say regarding generic absoluteness.}

As already mentioned, $\ZF+\DC$ is ``the right theory for classical analysis''. In this paper we argue that it is also the right basic theory for handling much of the forcing technology relevant to the real line. This argument is provided through the results in the first and second part of this paper. In the first part of the paper we observe that Dependent Choice is enough---in fact the minimum needed---of the Axiom of Choice to develop the theory of proper forcing, and that proper forcing must preserve Dependent Choice. We prove the consistency of the Proper Forcing Axiom ($\PFA$) by forcing, with a proper forcing notion, over any model of $\ZF+\DC$ in which there is a supercompact cardinal. We also show that assuming $\PFA$, a stronger version of Dependent Choice, namely $\DC_{\omega_1}$, must hold. As a consequence, all typical applications of $\PFA$ in $\ZFC$ at the level of $H(\omega_2)$ (see, for example, \cite{Baumgartner} or \cite{Moore:2005}) also go through in $\ZF+\DC$; in particular, $\PFA$ implies, in $\ZF+\DC$, that $2^{\aleph_0}=\aleph_2$ and that $H(\omega_2)$ has a definable well-ordering (\autoref{pfa-implies-wo}). We do not know how to prove that the Proper Forcing Axiom \textit{does not} imply the Axiom of Choice,\footnote{Without assuming the consistency, with $\ZF+\DC$, of the existence of a Reinhardt cardinal with a supercompact cardinal below it.} but we do suggest a possible way to prove such a statement assuming some generic absoluteness of the Proper Forcing Axiom (together with reasonable large cardinal assumptions). 

The second part of the paper begins with a (rather lengthy) discussion of the foundational significance of the generic absoluteness phenomenon. We then observe that generic absoluteness for the Chang model fails relative to forcing notions destroying $\DC$, and  move on to proving that if $\ZF+\DC$ holds and the supercompact cardinals form a stationary class in $\Ord$, then the theory of the Chang model is generically absolute with respect to forcing that preserves Dependent Choice (\autoref{gen-abs+}),\footnote{Monro showed in \cite{Monro:1983} that Dependent Choice is not always preserved by forcing in $\ZF$.}$^{,}$\footnote{In the $\ZFC$ context, the invariance of the theory of the Chang model under set theory in the presence of large cardinals (one supercompact cardinal, or a proper class of Woodin cardinals, is enough) is a well-known result of Woodin (\cite{Woodin:PNAS1988}).} and in fact with respect to $\DC$-preserving symmetric submodels of $\DC$-preserving forcing extensions (\autoref{symm-gen-abs}). 
We conclude the paper with some foundational remarks motivated by the proof of this last result.

\subsection{The basic definitions}
All our proofs will be carried out in $\ZF$. We will redefine the notion of $H(\kappa)$ to work in a context where the Axiom of Choice may fail. We then use this definition when discussing proper forcing. In our paper a notion of forcing is a partially ordered set $\PP$ with a maximum, denoted by $1_\PP$ (or just $1$). If $p,q\in\PP$, we say that $q$ \textit{extends} $p$, or that it is a \textit{stronger condition}, if $q\leq p$. If two conditions have a common extension, we say that they are \textit{compatible}. Rather than writing $\forces_\PP$ to denote the forcing relation corresponding to a notion of forcing $\PP$, we will omit the subscript and simply write $\forces$ whenever $\PP$ is clear from the context. A forcing $\PP$ is \textit{weakly homogeneous} if for every two conditions $p$ and $q$ there is an automorphism $\pi$ such that $\pi p$ and $q$ are compatible. If $\PP$ is weakly homogeneous, then for every $p\in\PP$ and every sentence $\varphi$ in the forcing language for $\PP$, $p\forces\varphi$ if and only if $1\forces\varphi$. Also, if $\PP$ and $\QQ$ are weakly homogeneous forcing notions, then so is $\PP\times\QQ$.

We denote $\PP$-names by $\dot x$, and if $x$ is a set in the ground model, we denote the canonical name for $x$ by $\check x$. If $\dot x$ and $\dot y$ are $\PP$-names, we say that \textit{$\dot y$ appears in $\dot x$} if there is some condition $p$ such that $\tup{p,\dot y}\in\dot x$. We then recursively define \[\dot x\restriction p=\{\tup{p_*,\dot y\restriction p}\mid\dot y\text{ appears in }\dot x,\,p_*\leq p,\, p_*\forces\dot y\in\dot x\}.\]

Not assuming the Axiom of Choice, the terminology, definitions and notations surrounding cardinals become vague. We consider \textit{cardinals} in the broader sense: if $x$ can be well-ordered, then $|x|$ is the least ordinal that has a bijection with $x$; otherwise, $|x|$ is the \textit{Scott cardinal} of $x$: \[|x|=\{y\mid\exists f\colon x\to y\text{ a bijection}\}\cap V_\alpha,\] for the least $\alpha$ for which the intersection is non-empty. Greek letters (with the notable exception of $\varphi$ and $\pi$) will denote ordinals, and specifically $\aleph$ numbers when we assume that they denote infinite cardinals.

We write $|x|\leq|y|$ to denote that there is an injection from $x$ into $y$, and $|x|\leq^*|y|$ to denote that $x$ is empty or else there is a surjection from $y$ onto $x$. The notation $|x|<|y|$ denotes the statement $|x|\leq |y|$ and $|y|\nleq|x|$, and $|x|<^*|y|$ is defined from $\leq^*$ similarly.\footnote{Note that $|x|<|y|$ is equivalent to $|x|\leq|y|$ and $|x|\neq|y|$, as the Cantor--Bernstein theorem is provable in $\ZF$. For $\leq^*$, however, the dual Cantor--Bernstein theorem is not provable without choice (see \cite{Jech:AC1973}, Chapter 11, Problem 8), and so $|x|<^*|y|$ is in fact a stronger statement than $|x|\leq^*|y|$ and $|x|\neq|y|$.}

Finally, for a set $x$, the \textit{Hartogs number of $x$} is $\aleph(x)=\min\{\kappa\mid\kappa\nleq|x|\}$ and the \textit{Lindenbaum number of $x$} is $\aleph^*(x)=\min\{\kappa\mid\kappa\nleq^*|x|\}$.
\section{Hereditary sets}
Structures of the form $H(\kappa)$, where $\kappa$ is a cardinal, are one of the key constructions in set theory. These are usually\footnote{Assuming choice, that is.} defined as $\{x\mid|\tcl(x)|<\kappa\}$. Assuming that $\kappa$ is a regular cardinal, these sets are models of $\ZFC^-$, namely all the axioms of $\ZFC$ except the Power Set Axiom (which can hold for ``small sets'' when $\kappa$ is sufficiently large). In almost all the definitions we have we can replace $H(\kappa)$ with $V_\alpha$ for ``sufficiently nice $\alpha$'', meaning one where enough of Replacement is satisfied. But this requires more care in noting what Replacement axioms are being used, and virtually nobody wants that. Not when a nice alternative like $H(\kappa)$ is available. 

In the absence of the Axiom of Choice, however, we want to have $H(\kappa)$ reflect the fact that choice fails. This cannot be done when taking the same definition at face value, since it implies that if $x\in H(\kappa)$ then $x$ can be well-ordered. This means that we need to modify the definition of $H(\kappa)$.
\begin{definition}
Given a set $x$, $H(x)$ is the class $\{y\mid|x|\nless^*|\tcl(y)|\}$. We denote by $H({<}x)$ the union of the classes $H(y)$ such that $|y|<^*|x|$. 
\end{definition}
Note that for an infinite cardinal $\kappa$, $H(\kappa)=H({<}\kappa^+)$ and if $\kappa$ is a regular limit cardinal, then $H({<}\kappa)=H(\kappa)$.
\begin{proposition}Let $\kappa>\omega$ be an uncountable cardinal. The following properties hold:
\begin{enumerate}
\item $H(\kappa)$ is a transitive set of height $\kappa$.
\item $H(\kappa)=V_\kappa$ if and only if $\kappa$ is a strong limit cardinal; in other words, $\kappa\nleq^*|V_\alpha|$ for all $\alpha<\kappa$. 
\item If $\kappa$ is regular, then $H(\kappa)\models\ZF^-$, and if it also a strong limit then $H(\kappa)\models\ZF_2$.
\end{enumerate}
\end{proposition}
\begin{proof}Clearly, $H(\kappa)$ is a transitive class. We want to argue that $H(\kappa)\subseteq V_\kappa$, which implies that $H(\kappa)$ is a set. For this we need the following claim: for a set $x$ the following are equivalent:\begin{enumerate}
\renewcommand{\theenumi}{\alph{enumi}}
\item The von Neumann rank of $x$ is $\alpha$.
\item $\alpha=\min\{\beta\mid \tcl(x)\cap V_{\beta+1}\setminus V_\beta=\varnothing\}$.
\end{enumerate}
In other words, if $x$ has rank $\alpha$, then its transitive closure contains elements of any rank below $\alpha$.

We prove this equivalence by $\in$-induction on $x$. Suppose that $\rank(x)=\alpha$. Then for all $\beta<\alpha$ there is some $y\in x$ such that $\alpha>\rank(y)\geq\beta$. By the induction hypothesis, $\tcl(y)\cap V_{\gamma+1}\setminus V_\gamma\neq\varnothing$ for all $\gamma<\beta$, and therefore $\tcl(x)\cap V_{\gamma+1}\setminus V_\gamma\neq\varnothing$; moreover $y\in\tcl(x)\cap V_{\beta+1}\setminus V_\beta$, so indeed $\alpha$ is the minimal ordinal satisfying (b). On the other hand, assuming (b) it is clear that $\rank(x)\geq\alpha$. But if $\rank(x)>\alpha$, then for some $y\in x$ the rank of $y$ is at least $\alpha$, in which case the above argument shows that $\tcl(x)\cap V_{\alpha+1}\setminus V_\alpha\neq\varnothing$, so $\alpha$ is not the minimal ordinal satisfying (b).

Using this equivalence we argue that if $\rank(x)>\kappa$, then $x\notin H(\kappa)$. This is clear, since $\tcl(x)$ can be mapped onto $\kappa$ by mapping each element to its rank. So indeed $H(\kappa)\subseteq V_\kappa$, and so $H(\kappa)$ is a transitive set.

If $\kappa$ is a strong limit cardinal, then by definition $V_\alpha\in H(\kappa)$ for all $\alpha<\kappa$, and equality follows. 

Finally, if $\kappa$ is regular, and $x\in H(\kappa)$, then every function from $x$ to $\kappa$ is bounded, in particular, if $f\colon x\to H(\kappa)$, then $f\in H(\kappa)$ as well, so $H(\kappa)$ satisfies second-order Replacement, and by transitivity of $H(\kappa)$ we get all the other axioms of $\ZF^-$. And of course, if $\kappa$ is also a strong limit cardinal, the above implies that $H(\kappa)=V_\kappa\models\ZF_2$.
\end{proof}
Unfortunately, despite what we said at the beginning of this section, $\ZF$ does not guarantee that there are any uncountable regular cardinals. This means, in turn, that we cannot guarantee that $H(\kappa)$ will satisfy Replacement axioms. Nevertheless, models of the form $H(\kappa)$ are somewhat more natural than models of the form $V_\alpha$, as they provide us with an ordinal structure which is slightly more ``reasonable'' and they will also satisfy Mostowski's collapse lemma (as opposed to $V_{\omega+\omega}$, for example, which will not). 
\begin{remark}
It is interesting to note that the requirement that $\kappa$ is an $\aleph$ is crucial for the proof that $H(\kappa)$ is even a set. The definition can be used with any cardinal $x$, but if $x$ is not well-ordered, then $|x|\nless^*\kappa$ for any ordinal $\kappa$, meaning that $\Ord\subseteq H(x)$. In the definition, we could have artificially cropped $H(x)$ when $x$ is not well-orderable, in order to ensure that $H(x)$ remains a set, but we want to somehow reflect the structure of the universe in a natural way and this cropping seems artificial and not justified by the mathematics at this point.
\end{remark}
\begin{question}Is there any particular use for $H(x)$ for an arbitrary $x$ in $\ZF$? Do the proper classes $H(x)$ all satisfy any particular weak set theory of interest?
\end{question}
\begin{proposition}
The class $\{H({<}\kappa)\mid\kappa\text{ is an }\aleph\text{-number}\}$ is a continuous filtration of the universe.
\end{proposition}
\begin{proof}
Trivially this class is a continuous sequence of transitive sets. If $x$ is any set, then taking $\kappa=\aleph^*(\tcl(x))$ implies that $x\in H(\kappa)=H({<}\kappa^+)$.
\end{proof}
\begin{remark}
It is consistent that $\aleph^*(\RR)=\kappa$ is a limit cardinal, in which case $\RR\in H(\kappa)$ but $\RR\notin H({<}\kappa)$. This is a consequence of the Axiom of Determinacy (see Exercise~28.16 in \cite{Kanamori}), for example, but the consistency of this statement does not require any large cardinals.\footnote{Adding $\beth_\omega=\kappa$ Cohen reals with automorphisms \`a la Cohen's first model, and taking the supports to be generated by sets of size strictly less than $\kappa$, the standard proof shows that there is no surjection from $\RR$ onto $\kappa$, but there is a surjection on any smaller ordinal.}
\end{remark}
\section{Dependent Choice}
\begin{definition}
Let $T$ be a tree, i.e.\ a well-founded partial order whose lower cones are chains. 
\begin{enumerate}
\item We say that $T$ is \textit{$\kappa$-closed} if for every $\alpha<\kappa$, every increasing sequence of nodes in $T$ indexed by $\alpha$ has an upper bound.
\item We say that $B\subseteq T$ is a \textit{branch} if it is a maximal chain, and a \textit{cofinal branch} is a branch which meets all the levels of the tree.
\end{enumerate}
\end{definition}
\begin{definition}
Let $\kappa$ be an infinite cardinal. We denote by $\DC_\kappa$ the following statement: Every $\kappa$-closed tree without maximal elements has a chain of order type $\kappa$. We use $\DC_{<\kappa}$ to abbreviate $\forall\lambda<\kappa,\DC_\lambda$, and for $\kappa=\aleph_0$ we simply write $\DC$.\footnote{$\DC$ has many known equivalents \cite[Form~43]{HowardRubin:1998}.}
\end{definition}
Among the basic consequences of $\DC$ we have the Axiom of Choice for countable families of non-empty sets, the countability of every countable union of countable sets, the regularity of $\omega_1$, and many others.
\begin{lemma}\label{lemma:WO trees}
Every well-orderable tree has a maximal chain. In particular, if $T$ is a $\kappa$-closed well-orderable tree without maximal elements, then $T$ has a chain of order type at least $\kappa$.
\end{lemma}
\begin{proof}
Enumerate $T$ and proceed by recursion.
\end{proof}
\begin{theorem}
The following are equivalent:
\begin{enumerate}
\item $\DC$.
\item The L\"owenheim--Skolem theorem for countable languages: every structure in a countable language has a countable elementary submodel.
\item For every $\alpha\geq\omega$ and every countable $A\subseteq V_\alpha$ there is a countable elementary submodel $M$ of $V_\alpha$ such that $A\subseteq M$.
\item For every $\alpha\geq\omega$ there is a countable elementary submodel $M\prec\tup{V_\alpha,\in}$.
\end{enumerate}
\end{theorem}
The observation that (1) and (2) are equivalent is easy to make, but is attributed to George Boolos, and appears in \cite{BoolosJeffery} as Exercise 13.3.\footnote{Other sources online direct to the third edition of the book, which was unavailable to the authors, but upon examining the fourth and fifth editions, this exercise seems to have been removed.}
\begin{proof}
The proof that (1) implies (2) is the usual proof of the L\"owenheim--Skolem theorem, noting that it does not use more than $\DC$. Namely, given a countable\footnote{`Countable' means of course injectable into $\omega$, and therefore in particular well-orderable. The fact that the relevant sets are well-orderable is used throughout this proof.} structure $M$ in a relational language $\cL$ (every constant and function can be seen as a relation), we define $N\subseteq M$ by recursion:
\begin{enumerate}
\item $N_0=\varnothing$.
\item Suppose that $N_k$ was defined and is countable. For every formula of the form $\exists y\varphi(y,\vec x)$, and every finite sequence $\vec n$ from $N_k$, choose $y_{\varphi,\vec n}$ such that $M\models\varphi[y_{\varphi,\vec n},\vec n]$. Note that there are only countably many formulae and countably many finite sequences, so we need to make countably many choices, which is doable by $\DC$. Then we define \[N_{k+1}=N_k\cup\{y_{\varphi,\vec n}\mid\varphi\text{ is a formula}, \vec n\in N_k^{|\vec n|}\}.\] This is again a countable set.
\end{enumerate}
Finally, let $N=\bigcup_{k\in\omega} N_k$. Then $N$ is a countable substructure in the language $\cL$. It is not hard to verify that $N\prec M$. We needed $\DC$, rather than just the axiom of choice for countable families of sets, since the definition of each $N_k$ depends on the previous choices made.

The implications $(2)\implies(3)$ and $(3)\implies(4)$ are easy and trivial, respectively. The last implication is obtained by noting that if $\DC$ fails, then there is some $\alpha$ such that $\DC$ fails in $V_\alpha$. Taking a countable elementary submodel $M$ of $V_\alpha$, there is some $T\in M$ such that $T$ is a counterexample to $\DC$. By elementarity, $T\cap M$ is a subtree of $T$ without maximal elements. Hence, since $T\cap M$ is countable, we may easily build an infinite branch of $T\cap M$ by recursion using a well-order on $T\cap M$ (\autoref{lemma:WO trees}). But this is a contradiction.
\end{proof}

Of course, in the above equivalence, we can replace $V_\alpha$ by $H(\aleph_\alpha)$, or any other filtration of the universe.
\part{CCC, properness, and forcing axioms without choice}
\section{Proper Forcing}
\begin{definition}
Let $\PP$ be a notion of forcing. We say that $\PP$ is \textit{proper} if for any sufficiently large $\kappa$\footnote{I.e., for a tail of $\kappa$.} and every countable elementary submodel $M\prec (H(\kappa),\in,\PP)$, if $p\in\PP\cap M$, then $p$ has an extension which is $M$-generic. Namely, there is some $q\leq p$ such that for any open dense $D\in M$, $D\cap M$ is predense below $q$.
\end{definition}

Without choice properness can be a bit quirky.
\begin{proposition}
The following are equivalent:
\begin{enumerate}
\item $\DC$.
\item $\Col(\omega,\omega_1)$ is not proper.
\item There exists a forcing which is not proper.
\end{enumerate}
\end{proposition}
\begin{proof}
Assume $\DC$ holds and let $\PP$ be a proper forcing. Suppose $\dot f$ is a name such that for some condition $p\in\PP$, $p\forces\dot f\colon\check\omega\to\check\omega_1$. Let $M$ be a countable elementary submodel of a large enough $H(\kappa)$ such that $\PP,p,\dot f\in M$. Since $M$ is an elementary submodel, $\omega\in M$, and therefore for each $n<\omega$, there is a predense set $D_n\in M$ whose conditions all decide the value of $\dot f(\check n)$, which by elementarity is an ordinal in $M$. Let $q\leq p$ be an $M$-generic condition. Then $q$ must force that the range of $\dot f$ must in fact lie in $M\cap\omega_1$. However, $\DC$ implies that $\omega_1$ is regular, and so this set is bounded. In particular, this means that $q$ forces that $\dot f$ has a bounded range and therefore is not surjective. In turn this means that for any name $\dot f$ there is no condition forcing that $\dot f$ is a surjection from $\omega$ onto $\omega_1$. Therefore $\DC$ implies that no proper forcing can collapse $\omega_1$ and, in particular, $\Col(\omega,\omega_1)$ is not a proper forcing.

The implication from (2) to (3) is trivial. Finally, if $\DC$ fails, then the definition of properness holds vacuously by the definition of ``sufficiently large $\kappa$'', thus making every forcing proper.
\end{proof}
\begin{remark}
We could have replaced ``sufficiently large'' by explicitly requiring $\kappa$ to be any fixed cardinal such that $H(\kappa)$ has both the power set of $\PP$ and the basic tools needed for defining the forcing relation. One could also define a hierarchy by saying that $\PP$ is ``proper below $\kappa$'', where $\kappa$ is fixed in the definition of properness, and then define proper as ``for all $\kappa$, $\PP$ is proper below $\kappa$''. However, we feel that the above is a good definition as it is flexible enough to allow for a certain degree of freedom in choosing $\kappa$ and it provides us with this nice characterization of $\DC$.

On the other hand, if one decides that it is best to require something explicit about $\kappa$, or stick to the ``properness below $\kappa$'' definition, then the failure of $\DC$ at $H(\omega_1)$ (e.g.\ if there is a Dedekind-finite set of reals) will imply that every forcing is proper, whereas this will not be the case if $\DC$ first fails at some point above $H(\omega_1)$. So we would be getting an odd equivalence of ``every forcing is proper'' given that choice of `properness'; avoiding this oddity is a good motivation for picking the definition of properness as we did above.
\end{remark}
\begin{proposition}\label{prop:sigma-closed-implies-proper}
Every $\sigma$-closed forcing is proper.
\end{proposition}
\begin{proof}
Let $\PP$ be a $\sigma$-closed forcing. If $M$ is a countable elementary submodel of some $H(\kappa)$ with $\PP\in M$, then enumerate the dense open sets in $M$ as $\{D_n\mid n\in\omega\}$ and enumerate $\PP\cap M$ as $\{p_n\mid n\in\omega\}$, and for any $p\in\PP\cap M$, recursively construct a decreasing sequence such that $q_0=p$ and $q_{n+1}$ is the least $p_k$ such that $p_k\in D_n$ and $p_k\leq q_n$. Since $\PP$ is $\sigma$-closed, there is some lower bound $q$ of all the $q_n$'s and it is clear that $q$ is an $M$-generic condition.
\end{proof}
One ambition might be to prove that every ccc forcing is proper, but without the Axiom of Choice chain conditions are tricky to even formulate. Mirna D\v{z}amonja suggested in private communication to the authors to follow Mekler's footsteps in \cite{Mekler:1984} and formulate ccc, at least in the presence of $\DC$, by saying that every condition is $M$-generic for any countable elementary submodel $M$. Of course, under this definition it is trivial that every ccc forcing is proper.
\begin{theorem}\label{thm:cohen-is-ccc}
Let $X$ be any set, and let $\Add(\omega,X)$ denote the forcing whose conditions are finite partial functions $p\colon X\times\omega\to 2$, ordered by reverse inclusion. Then $\DC$ implies that $\Add(\omega,X)$ is a ccc forcing, for any $X$.
\end{theorem}
Of course we can immediately assume that $X$ is not empty, otherwise this is an atomic forcing, which is entirely uninteresting (and it is certainly proper). This theorem is also immediately true for any well-orderable set $X$, as then $\Add(\omega,X)$ is also well-orderable and therefore admits maximal antichains and the standard $\ZFC$ arguments work with it without using choice. Nevertheless, we will only use the assumption that $X$ is not empty in the following proof.
\begin{proof}
Suppose that $M\prec H(\kappa)$ for any large enough $\kappa$ such that $X$, $\Add(\omega,X)$ and $D$ are all in $M$, where $D$ is a dense subset of $\Add(\omega,X)$. Let $p$ be any condition in $\Add(\omega,X)$, and let $p'=p\cap M\in M$. Then there is some $q\in D\cap M$ which is compatible with $p'$ by predensity of $D$. Moreover, $\dom (p\setminus p')\cap M=\varnothing$, so $p$ is also compatible with $q$. In other words, $\varnothing$, the maximum condition of $\Add(\omega,X)$, is $M$-generic, and so every condition is $M$-generic, for every $M$.
\end{proof}
\begin{theorem}\label{thm:proper-preserves-dc}
If $\PP$ is proper and $\DC$ holds, then $1\forces\DC$.
\end{theorem}
\begin{proof}
Let $\dot T$ be a $\PP$-name for a tree of height $\omega$ without maximal nodes and let $p\in\PP$. Let $M$ be a countable elementary submodel of some sufficiently large $H(\kappa)$ such that $\dot T$, $p$, $\PP\in M$. 

Let $q$ be an $M$-generic condition extending $p$ and let $\dot T_*=M\cap\dot T$. For each $n<\omega$, let $D_n$ be the collection of conditions $p\in \PP$ for which there is some $\PP$-name $\dot t$ such that $p\forces``\dot t$ is a node in $\dot T$ of level $n$''. For every $n$, $D_n$ is a dense subset of $\PP$ in $M$. Hence, $D_n\cap M$ is predense below $q$ for each $n$, and as a consequence $q$ forces that $\dot T_*$ is a subtree of $\dot T$ of height $\omega$. By a similar argument, we also have that $q$ forces $\dot T_*$ not to have any maximal nodes. But then, since of course $\dot T_*$ is forced to be countable, $q$ forces that $\dot T_*$ has a branch by \autoref{lemma:WO trees}. We have shown that every condition in $\PP$ can be extended to a condition adding a branch through $\dot T$, which means that $1\forces\DC$.
\end{proof}
The above theorem shows that one cannot violate $\DC$ with a proper forcing, which is somewhat similar to the fact that one cannot violate $\AC$ with any forcing. This fact provides an argument that $\ZF+\DC$ is a good base theory for working with proper forcing. This argument is strengthened by the following theorem.
\begin{theorem}
\begin{enumerate}
\item Two-step iterations of proper forcings are proper.
\item Countable support iterations of proper forcings are proper.
\end{enumerate}
\end{theorem}
\begin{proof}
If $\DC$ fails, then this holds vacuously. If $\DC$ holds, the usual proofs as given in the context of $\ZFC$ also work here, noting that all relevant countable choices in the $\ZFC$ proofs can equally be made in the present context; for a reference see Theorem~2.7 in \cite{AbrahamHB:2010}. However, the following caveat is needed here.

In $\ZFC$ we usually define the iteration $\PP\ast\dot\QQ$ as a partial order on the collection of ordered pairs of the form $\tup{p,\dot q}$, where $p\in\PP$ and $\dot q$ is a $\PP$-name such that $1\forces\dot q\in\dot\QQ$. We could replace this with $p\forces\dot q\in\dot\QQ$, and while for a two-step iteration this is entirely inconsequential, for a countable support iteration this matters a lot.

Moving from $p\forces\dot q\in\dot\QQ$ to $1\forces\dot q\in\dot\QQ$ usually requires the Axiom of Choice. But we can circumvent this as follows: if $\dot q$ is a name which appears in $\dot\QQ$ and $p\forces\dot q\in\dot\QQ$, then define $\dot q_*$ as $\bigcup\{1_\QQ\restriction p'\mid p'\perp p\}\cup\dot q\restriction p$, and easily $1\forces\dot q_*\in\dot\QQ$ and $p\forces\dot q_*=\dot q$.
\end{proof}
\section{The Proper Forcing Axiom}
\begin{definition}
The Proper Forcing Axiom ($\PFA$) states that if $\PP$ is a proper forcing, then for every collection $\cD=\{D_\alpha\mid\alpha<\omega_1\}$ of dense open subsets of $\PP$ there exists a $\cD$-generic filter.
\end{definition}
\begin{proposition}
$\PFA$ implies $\DC_{\omega_1}$.
\end{proposition}
\begin{proof}
Let $T$ be a $\sigma$-closed tree of height $\omega_1$ and with no maximal elements. Then by \autoref{prop:sigma-closed-implies-proper}, $T$ is proper. Consider $D_\alpha=T\setminus T\restriction\alpha$, for $\alpha<\omega_1$. Then $D_\alpha$ is a dense open set, so by $\PFA$ there is a generic filter meeting each $D_\alpha$, which is a branch in $T$.
\end{proof}
\begin{theorem}\label{pfa-implies-wo}
$\PFA$ implies that $2^{\aleph_0}=2^{\aleph_1}=\aleph_2$ and that there is a well-order of $H(\aleph_2)$ of length $\omega_2$ definable over $H(\aleph_2)$ from parameters.
\end{theorem}
\begin{proof}
First note that by $\DC_{\omega_1}$, we have that $\aleph_1\leq 2^{\aleph_0}$, and applying $\PFA$ to the binary tree $2^{<\omega}$ we get that $\aleph_1<2^{\aleph_0}$. Moreover, by $\DC_{\omega_1}$ we can choose a ladder system on $\omega_1$, i.e., a sequence $\tup{C_\alpha\mid\alpha<\omega_1\text{ is a limit ordinal}}$ such that each $C_\alpha$ is a cofinal subset of $\alpha$ of order type $\omega$. It then also follows that there is a sequence $\tup{S_\alpha\mid\alpha<\omega_1}$ of pairwise disjoint stationary subsets of $\omega_1$. 

Next we want to point out that Moore's proof that the Mapping Reflection Principle ($\MRP$) implies that $2^{\aleph_0}=2^{\aleph_1}=\aleph_2$ in \cite{Moore:2005} can be repeated under $\PFA$ even without choice. The following changes with respect to the proof as presented in Moore's paper are necessary:
\begin{enumerate}
\item Require that $\cf(\theta)>\omega_1$ in the statement of $\MRP$ (Definition~2.3).
\item Replace $H((2^{\omega_1})^+)$ with $H(\aleph^*(2^{\omega_1}))$ in Lemma~4.4.
\item Note that the existence of a ladder system and $\DC_{\omega_1}$ are enough for the proofs of Proposition~4.2 and Theorem~4.3 to go through.\qedhere
\end{enumerate}
\end{proof}

Thus, $\PFA$ implies $H(\omega_2)\models\AC$ and hence that all the richness of the usual $\ZFC$ combinatorics is available to this structure.

\begin{remark}
One could argue that the correct formulation of forcing axioms (and $\PFA$ specifically) should require that ${<}2^{\aleph_0}$ dense open sets admit a filter meeting them all, and add that $\aleph_1\neq 2^{\aleph_0}$ to avoid trivialities. In that case in models of $\ZF+\DC+(|x|<2^{\aleph_0}\rightarrow|x|\leq\aleph_0)$, e.g.\ in Solovay's model (see \cite{Solovay:1970}) or models of $\AD$, this type of axiom holds vacuously since by $\DC$ every countable family of dense open sets admits a filter meeting them, and every collection of strictly less than $2^{\aleph_0}$ sets is necessarily countable. We feel that this vacuity is a good reason not to adopt such a definition.

Of course, we can replace ${<}2^{\aleph_0}$ by ${<}^*2^{\aleph_0}$, namely require that there is a surjective map from the continuum onto the family of dense open sets but there is no surjection in the other direction. Since there is a surjection from the continuum onto $\omega_1$, and we require that $2^{\aleph_0}\neq\aleph_1$, there is no surjection in the other direction. Therefore formulating $\PFA$ this way implies that every collection of $\aleph_1$ dense open sets admits a generic filter. This is enough to prove $\DC_{\omega_1}$, and the above arguments then imply that $2^{\aleph_0}=\aleph_2$. 

Even weakening $\PFA$ to Martin's Axiom (assuming $\DC$ and defining ccc as ``every condition is $M$-generic for every countable model $M$''), the corresponding forcing axiom applied to $\Add(\omega,X)$, which is a ccc forcing by \autoref{thm:cohen-is-ccc}, would imply that if there is a surjection from the continuum onto $X$, then either there is a surjection from $X$ onto the reals or there is an injection from $X$ into the continuum. And in particular $\aleph_1<2^{\aleph_0}$.
\end{remark}

We finish this section with the outline of the consistency of $\PFA$ relative to a supercompact cardinal in $\ZF+\DC$. 

\begin{definition}[Woodin, Definition~220 in \cite{Woodin:2010}]\label{def-supercompact}
For an ordinal $\alpha$ we say that $\kappa$ is \textit{$V_\alpha$-supercompact} if there exists some $\beta>\alpha$ and an elementary embedding $j\colon V_\beta\to N$ such that:
\begin{enumerate}
\item $N$ is a transitive set and $N^{V_\alpha}\subseteq N$,
\item the critical point of $j$ is $\kappa$ (in particular $j$ is non-trivial),
\item $\alpha<j(\kappa)$.
\end{enumerate}
If $\kappa$ is $V_\alpha$-supercompact for all $\alpha$, we say that it is a \textit{supercompact} cardinal.
\end{definition}
Assuming the Axiom of Choice holds, this definition is of course equivalent to the standard definition by deriving fine and normal measures from the elementary embeddings.

\begin{proposition}\label{cons-PFA}
If $\kappa$ is a supercompact cardinal, then the countable support iteration of length $\kappa$ using lottery sums of counterexamples to $\PFA$ of minimal rank forces $\PFA$ in $\ZF+\DC$.
\end{proposition}
This approach to the consistency of $\PFA$ is due to Philipp Schlicht and is unpublished. It appears in Julian Sch\"older's masters thesis, \cite{Scholder}, where it is also extended. For the convenience of the reader, we sketch below the idea behind the proof.
\begin{proof}
We define an iteration $\PP_\alpha$ for $\alpha\leq\kappa$, which is a countable support iteration, and in which the $\alpha$th iterand is the lottery sum of all names of counterexamples to $\PFA$ of minimal rank ${<}\kappa$ (if these exist; otherwise it is the trivial forcing).

If $\PFA$ does not hold in the generic extension, let $\QQ$ be a rank-minimal counterexample and $\cD=\{D_\alpha\mid\alpha<\omega_1\}$ be a collection of dense open sets witnessing that $\PFA$ fails. Let $\dot\QQ$ and $\dot\cD$ be names for $\QQ$ and $\cD$, respectively, and let $p\in\PP_\kappa$ force that $\dot\QQ$ and $\dot\cD$ form a counterexample to $\PFA$ of minimal rank. Choose $\eta$ to be a large enough ordinal such that the above is true in $V_\eta$ and there is an elementary embedding $j\colon V_\eta\to N$ with critical point $\kappa$ and such that $j(\kappa)$ is above the rank of $\dot\QQ$ and $\dot\cD$, where $N$ is transitive and sufficiently closed. In particular we may assume that $\dot\QQ$, $\dot\cD\in N$. 

By the definition of our iteration and the closure of $N$, we may extend $p$, in $j(\PP_\kappa)$, to a condition $p'$ which, at stage $\kappa$, selects $\dot\QQ$ among the choices of the lottery at that stage. Let $H$ be a $j(\QQ)$-generic filter over $N$ such that $p'\in H$ and let $G=H\cap\PP_\kappa$. By a standard lifting argument, $j\colon V_\eta\to N$ lifts to an elementary embedding from $V_\eta[G]$ to $N[H]$, which we will also denote by $j$. Let $I$ be the $\dot\QQ^G$-generic filter over $V$ given by $H$. We may assume that $N[H]$ contains $j``V_\alpha$ for some $\alpha$ containing all relevant names of conditions in $\dot\QQ$, and hence that it contains $I':=j``I$. But the interpretation of $I'$ by $H$ generates a filter of $j(\dot\QQ)^H$ meeting all members of $j(\dot\cD)^H$. It follows that $p'$ does not force, in $j(\PP_\kappa)$ and over $N$, that $j(\dot\QQ)$ and $j(\dot\cD)$ form a counterexample to $\PFA$, and therefore neither does $p$. But then, by elementarity of $j$, $p$ does not force, in $\PP_\kappa$ and over $V_\eta$, that $\dot\QQ$ and $\dot\cD$ form a counterexample to $\PFA$, which is a contradiction.
\end{proof}

\section{Does PFA imply the Axiom of Choice?}

Probably not. In fact, if $\delta<\kappa$ are such that $\delta$ is supercompact and $\kappa$ is a Reinhardt cardinal, then forcing $\PFA$ via a forcing $\PP\subseteq V_\delta$ as in \autoref{cons-PFA} will preserve the Reinhardtness of $\kappa$.

We would have liked, of course, to produce a model of $\PFA$ without the Axiom of Choice starting from a more reasonable assumption (in the region of supercompact cardinals). However, this seems to be a much more difficult task than one would originally expect it to be. We believe that outlining this difficulty is constructive, and raises some interesting questions about the absoluteness of $\PFA$ even between models of $\ZFC$.

The ``obvious construction'' would be to start with a model of $\ZFC+\PFA$, and then construct a symmetric extension `\`a la the first model Cohen' (see Chapter~5 in \cite{Jech:AC1973}), using $\Add(\omega_2,\omega_2)$ as our forcing,\footnote{In this forcing conditions are partial functions $p\colon\omega_2\times\omega_2\to 2$ whose domain is of size at most $\aleph_1$, ordered by reverse inclusion.} thus obtaining an intermediate model which is closed under $\omega_1$-sequences. Using the K\"onig--Yoshinobu Theorem from \cite{KonigYoshinobu:2004}, the full generic extension must satisfy $\PFA$, so the intermediate model \textit{should too}.

However, when getting down to brass tacks, one sees that the question is this: If $\PP$ is a proper forcing, is $\PP$ also proper in a larger model which agrees on the same $\omega_1$-sequences?

This leads us to a very interesting question, which even in the context of $\ZFC$ has no obvious answer:

\begin{question}
Suppose that $\PP$ is an $\omega_2$-directed-closed forcing and $1\forces\PFA$. Does that mean that $\PFA$ holds in the ground model?
\end{question}

Of course, to make our example, we would need to argue with a proof in a model of $\ZF+\DC_{\omega_1}$, rather than a model of $\ZFC$. But we are optimistic that a proof in $\ZFC$ will also work---mutatis mutandis---in $\ZF+\DC_{\omega_1}$.

When approaching the above question, it seems that an easy way to solve it would be to prove the following statement: if $\QQ$ is an $\omega_2$-directed-closed forcing and $\PP$ is proper, then $1_\QQ\forces_\QQ``\check\PP$ is proper''. Yasuo Yoshinobu proved that even in $\ZFC$, this question admits a negative answer.

\begin{theorem}[Yoshinobu \cite{Yoshinobu:2019}; $\ZFC$]
Given a regular cardinal $\kappa>\omega$, there is a forcing of the form $\Add(\omega,1)\ast\Col(\omega_1,2^\kappa)\ast\dot\QQ$, where $\dot\QQ$ is a ccc forcing, which collapses $\omega_1$ after forcing with $\Add(\kappa,1)$. 
\end{theorem}

\part{Generic Absoluteness without Choice}

\section{Motivation and limitations}\label{Motivation and limitations} 
As mathematicians, we would like to have a foundation of mathematics $T$ which, apart from having other virtues, is powerful enough that it answers all, or most of, the questions we may be naturally interested in when doing mathematics. There is widespread agreement that set theory provides a universal language to do mathematics: we imagine an entity, $V$, that we think of as `the set-theoretic universe', i.e., the collection of all pure sets,\footnote{The intuition is that  `pure sets' should be collections consisting only of sets, all of which sets should themselves have this property, and so on. When trying to make sense of this intuition, we are naturally led to the idea of pure sets being precisely the members of the set-theoretic cumulative hierarchy.} and observe that every question in mathematics seems to be translatable into a question about what holds true in the set-theoretic universe. Hence, it is just natural to try to find such a foundation of mathematics phrased in the language of set theory. To be clear, this foundation of mathematics should be a first order theory $T$ in the language of set theory which, besides having the desired virtues---e.g., we would like to have some evidence that the description of $V$ given by $T$ is in fact consistent---is such that we can prove, for a variety of natural mathematical statements $\sigma$ we may potentially be interested in, that every model of $T$ is a model of $\sigma$ (or that every model of $T$ is a model of $\lnot\sigma$). 

The standard axiomatic system for set theory, $\ZFC$, has the virtues\footnote{About which we will not say anything here.} one would expect from a satisfactory foundation of mathematics. However, and since $\ZFC$ is computable and interprets Peano arithmetic, G\"{o}del's incompleteness theorem tells us that this theory is incomplete if it is consistent, and that in fact there are arithmetical statements that the theory does not decide. Moreover, any consistent strengthening of $\ZFC$ obtained by adding finitely many axioms to it will be affected by the incompleteness phenomenon in exactly the same way. 

Large cardinal axioms are a family of `natural' principles, intending to supplement $\ZFC$ and asserting the existence of some very high cardinals $\kappa$ with strong properties of some sort.\footnote{See \cite{Kanamori} for a lot of historical and mathematical information on these cardinals.} They tend to be phrased as reflection principles for $V$, or as principles which one can prove\footnote{Over conservative second order extensions of $\ZFC$.} to be equivalent to second order statements about $V$ asserting the existence of elementary embedding, different from the identity, from $V$ into some submodel $M$, and where $M$ has strong closure properties within $V$. These axioms tend to be linearly ordered, in the sense that for any two such axioms $\psi_0$ and $\psi_1$, it is usually the case that one of the corresponding theories, say $\ZFC+\psi_1$, proves the consistency of $\ZFC+\psi_0$, and in fact the existence of a rank-initial segment of the universe satisfying $\ZFC+\psi_0$, or vice versa.\footnote{On the other hand, by G\"{o}del's second incompleteness theorem, none of these theories proves its own consistency, provided they are in fact consistent.} Thus, these axioms form a hierarchy of more and more daring principles, stretching from just $\ZFC$, all the way up to inconsistency.\footnote{A natural closure point of this hierarchy, namely the existence of a nontrivial elementary from $V$ into $V$ itself, was in fact proposed by Reinhardt in the late 1960's, and proved, a few years later, to be inconsistent by Kunen (\cite{Kunen}).} Historically, the motivation behind large cardinal axioms has come from a variety of sources. In any case, a general underlying motivation for these axioms is the intuition that the set-theoretic universe ought to be very large. We should point out that here we are dealing with a truly open-ended family of principles.\footnote{This is in accord with their spirit: every notion of `largeness' should be strictly subsumed by an even stronger such notion.} In particular, there is no `official' definition of large cardinal axiom, but set-theorists surely recognize such axioms when they meet them. 

It seems to be an empirical fact that every consistent theory $T$ occurring in mathematical practice can be interpreted in $\ZFC+\psi$ for some large cardinal axiom $\psi$, where this is understood  at least in the weak sense of there being a large cardinal axiom $\psi$ such that $\ZFC+\psi$ proves the consistency of $T$.\footnote{This is of course not a mathematical assertion as it refers to mathematical practice and as we do not have any precise definition of `large cardinal axiom' (see above).} Moreover, given that large cardinal axioms tend to be linearly ordered in the way we mentioned, the sets of $\Pi^0_1$ sentences implied by consistent large cardinal axioms (tend to) form a linearly ordered sequence, under inclusion, of sets of true $\Pi^0_1$ sentences. These two points together suggest the appealing picture of an open-ended increasing sequence of $\Pi^0_1$-theories of arithmetic, each one of them characterized as the set of $\Pi^0_1$ theorems of $\ZFC$ together with some consistent large cardinal axiom, and whose union is the true $\Pi^0_1$-theory of $\mathbb N$. The suggestion that the $\Pi^0_1$-theory of $\mathbb N$ is actually given by the union of these $\Pi^0_1$-theories is justified by the contention that every true $\Pi^0_1$ sentence ought to be implied by some consistent reasonable theory $T$ possibly occurring in mathematical practice,\footnote{By `reasonable theory possibly occurring in mathematical practice' we mean a mathematical theory, compatible with general set-theoretic principles, and which mathematicians could possibly come up with.} together with the empirical fact, mentioned above, that if $\sigma$ is a $\Pi^0_1$ theorem of such a theory $T$, then there is a large cardinal axiom $\psi$ with the property that every model $M$ of $\ZFC+\psi$ contains a model $N$ of $T$. $N$ need not be an $\omega$-model from the point of view of $M$, but the fact that $\sigma$ holds in $N$ entails that it also holds in $M$. We could in fact use this picture as a guiding principle when thinking about the upper reaches of the large cardinal hierarchy.\footnote{Of course, the actual dividing line between those large cardinal axioms which are consistent and those which are not will always remain a sort of epistemological blindspot.} 

The bad news is that there are very natural mathematical questions beyond $\Pi^0_1$-arithmetic, the answer to which is independent from any genuine large cardinal axiom. The most prominent such question, already asked by Cantor, is perhaps the question whether $\CH$, i.e., $2^{\aleph_0}=\aleph_1$, is true or, more generally, the Continuum Problem, i.e., the question ``How many real numbers are there''?\footnote{Unlike $\CH$, all known independent arithmetical statements we have been referring to in the previous paragraph---G\"{o}del-style self-referential sentences, consistency statements, or for example the Paris--Harrington theorem---are either meta-mathematical statements about theories or, ultimately, artifacts designed to prove the incompleteness of some base theory.} Using Cohen's method of forcing one can prove that it is consistent, relative to $\ZFC$, to assume that $2^{\aleph_0}$ takes the value $\aleph_1$, or $\aleph_2$, or $\aleph_{\omega_1}$, or indeed $\aleph_\alpha$ for many other choices of $\alpha$. In fact, the same is true if we replace $\ZFC$ with $\ZFC$ together with any large cardinal axiom: one can, for example, force both $2^{\aleph_0}=\aleph_1$ and $2^{\aleph_0}=\aleph_2$ by a rather small forcing $\PP$, and all large cardinals there might be in the universe with retain their large cardinal properties in the  extension by $\PP$ (as was shown by Azriel Levy and Robert Solovay in \cite{LevySolovay}). 

Set-forcing or, more generally, class-forcing, is indeed an extremely powerful method for, starting with some model $M$ of (sufficiently many) axioms of $\ZFC$, producing another model of set theory satisfying some interesting statement (see \cite{Shelah:PIF}). This is effected by building a generic extension $M[G]\supseteq M$ of $M$. The other main method we have for producing models of set theory is to construct carefully chosen submodels of the universe (inner model theory). In fact, as things stand at present, the forcing method and the construction of inner models of set theory are essentially the only methods we have for proving independence results, over base set theories---of the form $\ZFC$ + $\psi$ for some large cardinal axiom $\psi$---about the infinite.\footnote{On the other hand, and as already mentioned, there is certainly independence in arithmetic (e.g.\  G\"{o}del sentences), about which forcing has nothing to say, at least directly. However, and as we already pointed out, at least $\Pi^0_1$ statements seem to be decided, in a coherent way, by suitable large cardinal axioms. And in any case arithmetic is after all the mathematical study of the finite.} It is therefore natural, from an empiricist standpoint, to seek to distil mathematical truth by isolating axioms---beyond $\ZFC$---which make reasonable portions of the truth predicate of the universe immune, to some extent, to the method of forcing. $V=L$ is just such an axiom, and in fact it completely rules out any intromission whatsoever of the method of forcing (this includes generic extensions via class-forcing); the reason being, of course, that no non-trivial forcing extension of $L$ will satisfy $V=L$. But $V=L$ is not regarded as a good axiom by most set-theorists, as it is a minimality hypothesis incompatible with most large cardinal axioms (e.g., if $V=L$, then there are no measurable cardinals). In fact, it is usually the case, for typical (canonical) inner models $M$, that the view that $V=M$ is seen to impose undesirable minimality constraints on the universe---most notably incompatibility with large cardinal axioms---that disqualify the answers to mathematical questions derived from $V=M$ as the true answers. This in effect leaves the construction of forcing extensions as the only method we feel we should worry about when it comes to deciding on the truth value---in the context of $\ZFC$ enhanced with large cardinal axioms---of statements about the infinite.\footnote{There is in principle another way to use the method of forcing to prove independence, namely the construction of symmetric submodels of forcing extensions. However, this move is not available when the base theory relative to which we are proving independence is $\ZFC$ (see the beginning of \autoref{concluding} for an elaboration on this).}

Arguably, the picture presented in the previous paragraph could change quite dramatically if Woodin's Ultimate-$L$ project eventually succeeds (see \cite{Woodin:Midrasha}). This project aims at constructing a canonical inner model $M$ for a supercompact cardinal, and such an inner model would in fact absorb most `classical' large cardinals there might exist in $V$ above the supercompact cardinal of $M$. If this inner model can in fact be constructed, then the axiom $V=\UltimateL$ would be compatible with at least most large cardinal hypotheses, and it would be immune to the method of forcing in the same way that the axiom $V=L$ is. We should point out that we already have axioms which are compatible with all consistent large cardinal axioms and which, when added to $\ZFC$, also completely neutralize the method of set-forcing. The assertion that there is a proper class of supercompact cardinals which are indestructible under directed closed forcing is one such axiom.\footnote{This was pointed out by Gunter Fuchs and Joel D.\ Hamkins.} Let us call this axiom $\A$. In fact, if $\kappa$ is a supercompact cardinal and $\mathbb P$ is a forcing notion of size less than $\kappa$, then $\kappa$ is not indestructible by ${<}\kappa$-directed closed forcing after forcing with $\PP$ (\cite{Hamkins}).\footnote{$\A$ does not neutralize class-forcing, though.} In any case, $V=\UltimateL$ would be incomparably more effective than an axiom like $\A$. Indeed, $V=\UltimateL$ would effectively answer many questions about the set-theoretic universe (e.g., it would imply that $2^{\aleph_0}=\aleph_1$ and in fact $\diamondsuit_{\omega_1}$ holds, etc.), whereas $\A$ has nothing to say about these facts: if there is a proper class of supercompact cardinals, then one can construct a class-sized partial order forcing $\A$ together with $\CH$, or with $2^{\aleph_0}=\aleph_2$, etc. Nevertheless, and however things turn out with $V=\UltimateL$, one can naturally object to this axiom, and of course also to axioms like $\A$, on the ground that they only neutralize set-forcing for vacuous reasons; namely, because the axiom fails in any non-trivial generic extension by set-forcing. That was of course also the case for the axiom $V=L$. We are not impressed with a strategy for arguing for the truth of some statement if this strategy consists in dogmatically banning any methods to show the independence of the statement, especially if these methods are compatible with the standard views about $V$ (e.g., maximality, etc.). 

The above methodological considerations justify the search for principles, $\psi$, which, when added to our standard foundation with large cardinals $T$, decide some significant fragment of the truth predicate of $V$ not in first order logic, but in the stronger logic naturally associated to the forcing method. In other words, we want $\psi$ to imply,\footnote{In first order logic.} for every member $\sigma$ of some reasonable class $\Sigma$ of sentences, that every forcing extension $M[G]$ of every model $M$ of $T$ satisfies $\sigma$ (or that every forcing extensions $M[G]$ of every model $M$ of $T$ satisfies $\lnot\sigma$). We then say that $\psi$ proves \emph{generic absoluteness} for $\Sigma$. Such principles $\psi$ would be natural axioms in that they would provide an effectively complete picture of $V$, at least for sentences in $\Sigma$, in the sense of this picture being complete with respect to, essentially, our only method for proving independence over $T$. Furthermore, it would be desirable that these principles do not render our strong logic degenerate, i.e., it would be desirable to have that $\psi$ holds in many generic extensions $M[G]$ of models $M$ of $T$.  

Fortunately we do have axioms which are compatible with all consistent large cardinal hypotheses, which imply some substantial amount of generic absoluteness, and which furthermore do so for interesting (non-dogmatic) reasons. In fact, as Woodin proved (\cite{Woodin:PNAS1988}), large cardinals themselves do exactly that; for example, the existence, in $\ZFC$, of a proper class of Woodin cardinals has exactly this effect, as it renders the theory of the inner model  $L(\mathbb R)$---and in fact of the Chang model (see \autoref{chang-model})---invariant under generic extensions by set-forcing, and as of course the axiom survives any set-forcing extension. Hence, large cardinals not only seem to decide first order arithmetic, in first order logic. They are in fact proved to also decide second order arithmetic, and actually the theory of the Chang model, in the logic of forcibility (stronger than first order logic but the next natural logic to give our attention to given our available methods). Moreover, this theory is already completely decided by the axiom asserting the existence of a proper class of Woodin cardinals.  This fact provides a powerful additional argument in favour, in the context of $\ZFC$, of large cardinal axioms, at least in the region of a proper class of Woodin cardinals.    

What we have said so far should be enough to motivate those subscribing to the universist approach, namely the point of view that there is one real universe, which our axioms should aim at describing. But generic absoluteness is important for those subscribing to the pluralist or the multiversist approaches as well. In this type of approaches, one shifts the focus of attention from the notion of one unique universe of set theory to a plurality of competing universes of set theory, naturally generated by applying (some of) our methods for proving independence over a standard foundation. The independence phenomenon has made the actual set-theoretic universe into an elusive, or inaccessible, entity. Hence, set-theorists have been effectively forced to familiarizing themselves with the corresponding multiverse(s), which after all are the object they can access. Given the prevalence of the forcing method in set theory, the \emph{generic multiverse}, i.e., the collection of universes of set theory generated from a given universe by iterating the operations of taking set-generic extensions and ground models, has been the main object of attention in multiverse set theory. It is a remarkable fact about the generic multiverse that if we start with a model $M$ of $\ZFC$ and build a set-forcing extension $M[G]$ of it, the ground model $M$ is definable in $M[G]$ from some parameter. In fact, we can uniformly define all the ground models, from parameters, and in fact all the ground models of all generic extensions. Hence, we can simulate talk about the generic multiverse within any of its members $M$; for example, given any $n$, we can define $\Sigma_n$-truth in the generic multiverse, i.e., the $\Sigma_n$-sentences which are true in all members of the generic multiverse, inside $M$ (see \cite{Woodin:Multiverse} for more details). 
  
The generic multiverse both provides a concrete mathematical arena to debate foundational views in mathematics\footnote{For example, a view as the one expressed by Feferman in \cite{Feferman}, that a statement like the Continuum Hypothesis does not have a definite truth, can be naturally argued for using the language provided by the generic multiverse; in this case, the argument would, at least in part, appeal to the fact that neither $\CH$ nor its negation are true across the generic multiverse truth, i.e., each of $\CH$ and $\lnot\CH$ is true in some members of the generic multiverse and false in others.} and, because of its foundational value, it has also become a mathematical object of interest. To cite an example of mathematical work done on the generic multiverse with an eye on foundations, Hamkins and L\"owe have studied the modal logic of forcing, where this is understood as the modal logic corresponding to looking at the generic multiverse as a Kripke frame where the accessibility relation is given by `$N$ is a set-generic extension of $M$' (see \cite{HamkinsLoewe:ForcingModalLogic}).

Using the generic multiverse language, the generic absoluteness phenomenon can be used to provide a bridge between various possible positions in the universist vs.\ multiversist debate in the foundations of set theory. The first illustration of this that comes to mind is the following: As we mentioned, the presence, in some members of the the generic multiverse (equivalently, in all of its members), of a proper class of sufficiently strong large cardinals (e.g.\ Woodin cardinals) implies that the truth of the Chang model is invariant across the generic multiverse, whereas this is not the case for $\CH$. This suggests seeing the theory of the Chang model as `more definite' or `truer' than the theory of, say, $H(\omega_2)$ (where $\CH$ lives). That in turn suggests seeing the Chang model as `more real' than $H(\omega_2)$. According to such a view, reality about the set-theoretic universe would come in degrees, with a fragment of the universe being more real than another if there is less variation of the theory of the former, across the generic multiverse, than of the theory of the latter.\footnote{The perception that $H(\omega_1)$ is `more real than' $H(\omega_2)$ seems in fact to be widespread among mathematicians working outside of set theory.}
 
It is now a good time to define Chang models in general.

\begin{definition}\label{chang-model}
Let $\kappa$ be an ordinal. The \textit{$\kappa$-Chang model}, denoted by $\cC_\kappa$, is the $\subseteq$-minimal model of $\ZF$ containing all the ordinals and closed under $\kappa$-sequences. This model can always be constructed as $L(\power_{\kappa^+}(\Ord))$. In the case where $\kappa=\omega$, we omit $\omega$ from the terminology and notation, and refer to it as the Chang model.
\end{definition}

In the following section we will prove that under suitable large cardinal assumptions, the theory of the Chang model is generically absolute under forcings which preserve $\DC$. In other words, assuming sufficiently strong large cardinals exist, if $\PP$ is a forcing which preserves $\DC$, e.g.\ a proper forcing, and $G\subseteq\PP$ is a $V$-generic filter, then the theory of the Chang model in $V$ and $V[G]$ is the same, even if we add reals as parameters. The foundational motivation for this work is the following: As has already been mentioned, most of real analysis lives inside $L(\RR)$, and therefore in the Chang model,\footnote{Note that $L(\RR)$ is in fact a definable inner model of the Chang model.} and can be carried out in $\ZF+\DC$. This means that $\ZF+\DC$ is a very natural meta-theory for analysis. On the other hand, we have made a case for the desirability of generic absoluteness. These two points together make it natural to enquire whether or not full $\AC$ is an overkill when it comes to deriving generic absoluteness for $L(\RR)$ or even for the Chang model (modulo large cardinals), in the same way that, say, $\CH$ or $\lnot\CH$ are indeed overkills when deriving generic absoluteness for these inner models with $\ZFC$ as base theory.

It is important to realize that non-trivial instances of generic absoluteness for the Chang model may be blocked unless we restrict our considerations to forcing notions preserving $\DC$:

\begin{proposition}
The Chang model cannot have generic absoluteness for its $\Sigma_2$ theory in $\ZF$, even in the presence of large cardinals. 
\end{proposition}
\begin{proof}
Suppose that $V$ is a model of $\ZFC$ and $\kappa$ is a supercompact cardinal (or extendible, or any other not-yet inconsistent large cardinal). Perform a Feferman--Levy style construction taking the symmetric collapse of $\Col(\omega_1,{<}\aleph_{\omega_1})$ (see Section~4 in \cite{HayutKaragila:2017} for details). 

This provides us with a model $M$ of $\ZF+\DC$ where $\kappa$ is still a supercompact cardinal in the sense of \autoref{def-supercompact},\footnote{\label{footnote-Levy-Solovay}The fact that the elementary embeddings from the ground model lift follows from Theorem~4.9 in \cite{HayutKaragila:2018}. While the proof there does nothing to prove the needed closure properties of these lifted embeddings, this is not hard to verify by hand.} and $M\models\cf(\omega_2)=\omega_1$; moreover any countable sequence of ordinals in $M$ belongs to the ground model $V$. Therefore the Chang model $\cC$ in $M$ is the same as in $V$.

Working in $M$, let $G$ be an $M$-generic filter for $\Col(\omega,\omega_1)$. Then in $M[G]$ we have that $\omega_1^{M[G]}=\omega_2^M$ and therefore $M[G]\models\cf(\omega_1)=\omega$. This fact is reflected to $\cC$ since it is witnessed by a countable sequence of ordinals, despite the fact that $\Col(\omega,\omega_1)$ has size $\aleph_1$ (without using any choice!) and thus is strictly less than $\kappa$. However, $\cf(\omega_1)=\omega$ is a $\Sigma_2$ statement over the Chang model.
\end{proof}

\begin{remark} The above proof shows of course that if $\kappa$ is, say, a supercompact cardinal in the $\ZFC$ model $V$, then there is a symmetric extension $M$ satisfying $\DC$ in which $\kappa$ remains supercompact and such that $\Sigma_2$-generic absoluteness for the Chang models fails between $M$ and the $\Col(\omega, \omega_1)$-extension of $M$ (which in the above situation kills $\DC$). Hence, there is no hope for proving $\Sigma_2$-generic absoluteness for the Chang model in general over the base theory $\ZF+\DC$ unless we restrict to forcings preserving $\DC$, e.g.\ proper forcings.\end{remark}

\begin{question}
Can generic absoluteness for $L(\RR)$ fail, in the presence of large cardinals, without the requirement that $\DC$ is preserved? Moreover, is it at all possible, assuming the existence of $\RR^\#$, that $\omega_1$ be singular in $L(\RR)$?
\end{question}

\section{Generic absoluteness for the Chang model}

Given an infinite regular cardinal $\lambda$ and a non-empty set $X$, $\Col(\lambda,{<}X)$ is the forcing, ordered by reverse inclusion, of all functions $p$ such that $|\dom p|<\lambda$, $\dom p\subseteq\lambda\times X$, and $p(\alpha,x)\in x$ for all $\alpha<\lambda$ and $x\in X$. It is straightforward to see that $\Col(\omega,{<}X)$ is weakly homogeneous and that forcing with $\Col(\omega,{<}X)$ adds surjections from $\omega$ onto all members of $X$. 

The following technical lemma, \autoref{forcing-AC}, is a rendering of \cite{Woodin:2010}, Theorem 226. This lemma will be crucially used in the proof of \autoref{gen-abs-lemma} (and of \autoref{gen-abs-lemma0}).

\begin{lemma}[Woodin; $\ZF+\DC$]\label{forcing-AC} Suppose $\kappa$ is a supercompact cardinal and $\PP\in V_\kappa$ is a forcing notion preserving $\DC$. There is then a strictly increasing sequence $\tup{\kappa_\xi\mid\xi<\kappa}$ of strongly inaccessible cardinals below $\kappa$, together with a sequence $\tup{\cQ_\xi\mid 1\leq\xi\leq\kappa}$ of forcing notions and a sequence $\tup{\dot\cQ^\PP_\xi\mid 1\leq \xi\leq\kappa}$ of $\PP$-names for forcing notions, satisfying the following conditions.
\begin{enumerate}

\item $\kappa_0$ is such that $\PP\in V_{\kappa_0}$ and such that 
\begin{itemize}
\item $\Col(\omega_1,{<}V_{\kappa_0})$ forces $\DC_{{<}\kappa_0}$ over $V_{\kappa}$, and
\item $\PP$ forces that $\Col(\omega_1,{<}V_{\kappa_0}[\dot G_\PP])$ forces $\DC_{{<}\kappa_0}$ over $V_{\kappa}[\dot G_\PP]$.
\end{itemize}
Also, 
\begin{itemize}
\item $\cQ_1=\Col(\omega_1,{<}V_{\kappa_0})$, and 
\item $\dot{\cQ}^\PP_1$ is a $\PP$-name for $\Col(\omega_1,{<}V_{\kappa_0}[\dot G_\PP])$.
\end{itemize}

\item For every $\xi>0$, letting $\gamma_\xi =\sup\{\kappa_{\xi'}\mid\xi'<\xi\}$, $\kappa_\xi$ is such that 
\begin{itemize}
\item $\cQ_\xi$ forces that $\Col(\gamma_\xi,{<}V_{\kappa_\xi}[\dot G_{\cQ_\xi}])$ forces $\DC_{{<}\kappa_\xi}$ over $V_{\kappa}$, and
\item $\PP\ast\dot{\cQ}_\xi^\PP$ forces that $\Col(\gamma_\xi,{<}V_{\kappa_\xi}[\dot G_{\PP\ast\dot{\cQ}^\PP_\xi}])$ forces $\DC_{{<}\kappa_\xi}$ over $V_{\kappa}[\dot G_\PP]$.
\end{itemize}
Also, 
\begin{itemize}
\item $\cQ_{\xi+1}=\cQ_\xi\ast\Col(\gamma_\xi,{<}V_{\kappa_\xi}[\dot G_{\cQ_\xi}])$, and
\item $\dot{\cQ}^\PP_{\xi+1}$ is a $\PP$-name for $\dot{\cQ}^\PP_\xi\ast\Col(\gamma_\xi,{<}V_{\kappa_\xi}[\dot G_{\dot{\cQ}^\PP_\xi}])$.
\end{itemize}
\item For every limit $\xi$, $\cQ_\xi$ is the Easton limit of $\tup{\cQ_{\xi'}\mid\xi'<\xi}$ and $\dot\cQ_\xi^\PP$ is a $\PP$-name for the Easton limit of $\tup{\dot\cQ_{\xi'}^\PP\mid\xi'<\xi}$.
\item For every $\xi<\kappa$, if $H\subseteq \cQ_\xi$ is $V$-generic (resp., if $H'\subseteq \cQ_\xi^\PP$ is $V[G]$-generic), then $\cQ_\kappa/H$ is $\kappa_\xi$-closed in $V[H]$ (resp., $\cQ_\kappa^\PP/H$ is $\kappa_\xi$-closed in $V[G][H']$).
\item For every $\delta<\kappa$ such that $\PP\in V_{\delta}$, if $\delta$ is supercompact in $V$ (in $V^\PP$), then $\kappa_\delta=\delta$ and $\delta$ remains supercompact after forcing with $\cQ_\kappa$ over $V$ (after forcing with $\PP\ast\dot{\cQ}^\PP_\kappa$ over $V$).
\item If $H\subseteq\cQ_\kappa$ is $V$-generic, then $V_\kappa[H]\models\ZFC$. 
\item If $G\subseteq\PP$ is $V$-generic and $H\subseteq\cQ^\PP_\kappa$ is $V[G]$-generic, then $V_\kappa[G][H]\models\ZFC$. 
\end{enumerate}
\end{lemma}

\autoref{gen-abs-lemma} and its companion, \autoref{gen-abs-lemma0}, are the key ingredients in the proof of \autoref{gen-abs}.

\begin{lemma}[$\ZF+\DC$]\label{gen-abs-lemma}
Suppose $\delta<\kappa$ are supercompact cardinals, $\PP\in V_\delta$ is a forcing notion preserving $\DC$, and $G\subseteq\PP$ is a $V$-generic filter. Then there is, in some outer model, a $V[G]$-generic filter $K\subseteq\Col(\omega,{<}V_\delta)$, together with an elementary embedding \[j\colon\cC^{V_\kappa[G]}\to\cC^{V_\kappa[G][K]}\] \end{lemma}

\begin{proof}
Let $G\subseteq\PP$ be a $V$-generic filter and, in $V[G]$, let $\tup{\cQ_\xi^\PP\mid\xi\leq\kappa}$ be the iteration given by \autoref{forcing-AC}. Let $H\subseteq\cQ^\PP_\kappa$ be a $V[G]$-generic filter. By \autoref{forcing-AC} (5) and (7), $\delta$ remains supercompact in $V_\kappa[G][H]$ and $V_\kappa[G][H]\models\ZFC$. Hence, by a classical $\ZFC$ result of Woodin (see \cite{Woodin:PNAS1988}) there is, in some outer model $N$, a $V_\kappa[G][H]$-generic filter $J\subseteq\Col(\omega,{<}\delta)$, for which there is an elementary embedding \[j\colon\cC^{V_\kappa[G][H]}\to\cC^{V_\kappa[G][H][J]}.\] Let $X$ be $V_\delta$ as computed in $V$. Since $\Col(\omega,{<}X)$ and $\Col(\omega, {<}\delta)$ are forcing-equivalent in $V_\kappa[G][H]$, we may fix a $V_\kappa[G][H]$-generic filter $K_0\subseteq\Col(\omega, {<}X)$ such that $\cC^{V_\kappa[G][H][J]}=\cC^{V_\kappa[G][H][K_0]}$.

\begin{claim}
There is some $V_\kappa[G]$-generic filter $K\subseteq\Col(\omega, {<}X)$ such that \[\cC^{V_\kappa[G][H][K_0]}=\cC^{V_\kappa[G][K]}.\]
\end{claim}

\begin{proof} Let $H_\delta=H\cap\cQ_\delta^\PP$. Since $\cQ_\kappa^\PP/H_\delta$ is $\delta$-closed in $V[G][H_\delta]$ (\autoref{forcing-AC} (4) and (5)), each $\Col(\omega,{<}X)$-name  in $V_\kappa[G][H]$ for an $\omega$-sequence of ordinals is equivalent to such a name in $V_\kappa[G][H_\delta]$: every such name is equivalent to the canonical name $\dot x$ for an $\omega$-sequence where, for each $n<\omega$, $\dot x(n)$ consists of pairs $\langle p, \check \alpha\rangle$, for $p$ a condition in $\Col(\omega, {<}X)$ and $\alpha$ an ordinal, and where the set of such $p$'s forms an antichain of $\Col(\omega, {<}X)$ and is therefore of size less than $\delta$. Hence $\cC^{V_\kappa[G][H][K_0]}=\cC^{V_\kappa[G][H_\delta][K_0]}$. But $\cQ_\delta^\PP$ clearly embeds completely in $\Col(\omega, {<}X)\times\Col(\omega, {<}X)$, and therefore also in $\Col(\omega, {<}X)$ since $\Col(\omega, {<}X)\cong\Col(\omega, {<}X)\times\Col(\omega, {<}X)$, from which the conclusion follows.  
\end{proof}

The above claim finishes the proof since $\cQ_\kappa^\PP$ is $\sigma$-closed in $V[G]$, by \autoref{forcing-AC} (4), and therefore $\cC^{V_\kappa[G][H]}=\cC^{V_\kappa[G]}$. 
\end{proof}

Similarly, we can prove the following lemma.

\begin{lemma}[$\ZF+\DC$]\label{gen-abs-lemma0}
Suppose $\delta<\kappa$ are supercompact cardinals. Then there is, in some outer model, a $V$-generic $K\subseteq\Col(\omega,{<}V_\delta)$, together with an elementary embedding $$j\colon\cC^{V_\kappa}\to\cC^{V_\kappa[K]}$$
\end{lemma}

The main theorem in this section is the following.

\begin{theorem}[$\ZF+\DC$]\label{gen-abs} Suppose for every closed and unbounded $\Pi_2$-definable class $C$ of ordinals there is a supercompact cardinal $\kappa\in C$. Then, for every set-forcing $\PP$ and every $V$-generic filter $G\subseteq\PP$, if $V[G]\models\DC$, then the structures $(\cC^V; \in, r)_{r\in \RR^V}$ and $(\cC^{V[G]}; \in, r)_{r\in \RR^V}$ have the same $\Sigma_2$-theory.
\end{theorem}

\begin{proof} Suppose, towards a contradiction, that $\PP$ forces that $(\cC^V; \in, r)_{r\in \RR^V}$ and $(\cC^{V[\dot G]}; \in, r)_{r\in \RR^V}$ disagree on the truth value of some sentence $\exists x\forall y\varphi(x, y)$, where $\varphi$ is a restricted formula. By our large cardinal assumption we may then fix a supercompact cardinal $\kappa$ such that 
 \[(\cC^{V_\kappa}; \in, r)_{r\in \RR^V}\models\exists x\forall y\varphi(x,y)\iff(\cC^{V_\kappa[G]}; \in, r)_{r\in \RR^{V}}\models \lnot\exists x\forall y\varphi(x,y)\] and such that there is a supercompact cardinal $\delta<\kappa$ such that $\PP\in V_\delta$.
 
We show that $(\cC^{V_{\kappa}}; \in, r)_{r\in\RR^V}$ and $(\cC^{V_{\kappa}[G]}; \in, r)_{r\in \RR^{V}}$ are elementarily equivalent, which will be a contradiction. 

Let $X$ be $V_\delta$ as computed in $V$. By \autoref{gen-abs-lemma} and \autoref{gen-abs-lemma0} we know that there are, in some outer model, filters $K,K'\subseteq\Col(\omega,{<}X)$ which are $V$-generic and $V[G]$-generic, respectively, for which there are elementary embeddings \[j\colon\cC^{V_{\kappa}}\to \cC^{V_{\kappa}[K]}\ \text{ and }\ j'\colon\cC^{V_{\kappa}[G]} \to \cC^{V_{\kappa}[G][K']}.\] Since $\PP\times\Col(\omega,{<}X)$ and $\Col(\omega, {<}X)$ are forcing--equivalent, there is a $V$-generic filter $K''\subseteq\Col(\omega, {<}X)$ such that $V_\kappa[G][K']=V_\kappa[K'']$. But then, by the weak homogeneity of $\Col(\omega,{<}X)$, the theories of $(\cC^{V_{\kappa}}; \in, r)_{r\in\RR^V}$ and $(\cC^{V_{\kappa}[G]}; \in, r)_{r\in \RR^{V}}$ are the same.
\end{proof}

Using a similar argument one can prove the following.

\begin{theorem}[$\ZF+\DC$]\label{gen-abs+}
Suppose there is, for every closed and unbounded class of ordinals $C$, a supercompact cardinal $\kappa$ such that $\kappa\in C$. Then, for every set-forcing $\PP$ and every $V$-generic filter $G\subseteq\PP$, if $V[G]\models\DC$, then the structures $(\cC^V; \in, r)_{r\in \RR^V}$ and $(\cC^{V[G]}; \in, r)_{r\in \RR^V}$ are elementarily equivalent.
\end{theorem}

Note the second order character of the hypothesis of \autoref{gen-abs+}. This hypo\-the\-sis can of course be taken to be an infinite scheme asserting that for every formula $\Theta(x)$ with parameters defining a closed and unbounded class of ordinals there is some supercompact cardinal $\kappa$ such that $\Theta(\kappa)$. Many reasonable large cardinal assumption could be used here in place of this. Nevertheless, it is not clear to us whether the existence of one supercompact cardinal suffices to yield the conclusion. 

The following version of \autoref{gen-abs} for $L(\RR)$ can be derived by the same argument as in the proof of \autoref{gen-abs} together with the fact that, under the existence of a supercompact cardinal, $\RR^\#$ exists in any set-forcing extension, and together with Woodin's classical result that the Axiom of Determinacy holds in $L(\RR)$ assuming $\ZFC$ and the existence of infinitely many Woodin cardinals below a measurable cardinal.

\begin{theorem}[$\ZF+\DC$] Suppose there are two supercompact cardinals. Then, for every set-forcing $\PP$ and every $V$-generic filter $G\subseteq\PP$-generic filter $G$ over $V$, if $V[G]\models\DC$, then

\begin{enumerate}
\item the structures $(L(\RR)^V; \in, r)_{r\in \RR^V}$ and $(L(\RR)^{V[G]}; \in, r)_{r\in \RR^V}$ are elementarily equivalent, and
\item the Axiom of Determinacy holds in $L(\RR)^{V[G]}$.
\end{enumerate}
\end{theorem}

\section{Some concluding remarks}\label{concluding} The construction of symmetric submodels of forcing extensions---also known as symmetric extensions---is another way to use forcing to prove independence over base theories $T$.  Such a symmetric submodel is a model $N$ such that $M\subseteq N\subseteq M[G]$, for a ground model $M$ satisfying our base theory and a generic filter $G$ over $M$.\footnote{It should be remarked that not every such intermediate model is a symmetric extension, as was shown by the second author in \cite{Bristol}.} When $T$ is $\ZFC$, we of course do not get anything new by looking at symmetric extensions.  The reason is that if $M$ and $N$ both satisfy $\ZFC$, then $N$ is itself a forcing extension of $M$. On the other hand, if we drop the Axiom of Choice from our foundation, then the construction of symmetric extensions becomes a genuine new method for proving independence. In particular, retreating to a foundation which does not incorporate Choice allows for a richer and more complex generic absoluteness project; now we can also ask if such and such hypothesis implies that all symmetric extensions of $V$ (of some nice form) agree with $V$ on the theory of the Chang model, or we can ask \autoref{question-abs} below, etc. 

In this context, one first outcome of foundational significance from our result is that, in the presence of sufficiently strong large cardinals, the theory of the Chang model is invariant not also under $\DC$-preserving forcing, but also under $\DC$-preserving symmetric submodels of $\DC$-preserving forcing extensions.\footnote{See also \cite{DC} for a discussion about $\DC$-preserving symmetric extensions.}

\begin{theorem}\label{symm-gen-abs}
Suppose that $M$ satisfies $\ZF+\DC$ and the hypothesis of \autoref{gen-abs+}. Let $M[G]$ be a generic extension of $M$ such that $M[G]$ also satisfies $\DC$, and let $N$ be a corresponding $\DC$-preserving symmetric extension. Then $\cC^N$ and $\cC^M$ are elementarily equivalent.
\end{theorem}
\begin{proof}
As we remarked in footnote~\ref{footnote-Levy-Solovay}, the relevant large cardinal hypothesis  is preserved when moving from $M$ to $N$.
It follows, therefore, that $N$ will satisfy the hypothesis of \autoref{gen-abs+}. By a theorem of Grigorieff in \cite{Grigorieff:1975}, $M[G]$ is a generic extension of $N$. As we assumed that all three models satisfy $\DC$, this is a $\DC$-preserving forcing extension, so by \autoref{gen-abs+}, $\cC^N$ and $\cC^{M[G]}$ are elementarily equivalent, as well as $\cC^M$ and $\cC^{M[G]}$. As elementary equivalence is, well, an equivalence relation, it follows that $\cC^M$ and $\cC^N$ are also elementarily equivalent.
\end{proof}

This observation leads to the following question.

\begin{question}\label{question-abs}
Suppose that $M$ satisfies $\ZF+\DC$ and that the theory of the Chang model is generically absolute under $\DC$-preserving set-forcing. Let $N$ be a $\DC$-preserving symmetric extension of $M$. Is $\cC^M$ elementarily equivalent to $\cC^N$?
\end{question}

It should be noted that generic absoluteness, in general, does not imply the existence of a proper class of large cardinals in the ground model. In fact, it does not imply even the existence of an inaccessible cardinal. For example, suppose $V$ is a model of $\ZFC$ in which there is a supercompact cardinal $\kappa$ but there are no inaccessible cardinals above $\kappa$. Suppose $G$ is a $\Col(\omega_1, {<}\kappa)$-generic filter over $V$. Then generic absoluteness for the Chang model holds in $V[G]$ but in $V[G]$ there are no inaccessible cardinals.  

The proofs of our main results in the previous section can be seen as an illustration of the thesis that, in choiceless contexts, one obtains fragments of choice from sufficiently strong large cardinal axioms (see \cite{Woodin:2010}). Roughly speaking, they show that if $\DC$ holds and  a large enough supply of supercompact cardinals is available, then one can force $\AC$, over suitable initial segments $V_\alpha$ of the universe, while preserving large cardinals by a forcing---and here is a crucial point---which does not add $\omega$-sequences of ordinals and therefore leaves the corresponding versions of the Chang model unchanged. One can then appeal to the classical arguments for proving generic absoluteness for this inner model over $\ZFC$. The reason we needed to force full choice, at least over some suitable $V_\alpha$, is that in the classical arguments for proving generic absoluteness at this level---which we ultimately appeal to---we need {\L}o\'{s}'s Theorem, between $V$ and some generic ultrapower of $V$, and {\L}o\'s's Theorem is not provable in $\ZF+\DC$.\footnote{By the work of Howard in \cite{Howard:1975}, $\AC$ is equivalent to the conjunction of {\L}o\'s's Theorem and the Boolean Prime Ideal theorem, and by \cite{Pincus:1977} the Boolean Prime Ideal theorem is consistent with $\DC$ together with the failure of $\AC$, and thus $\ZF+\DC$ does not prove {\L}o\'s's Theorem.}

The extra large cardinal strength---with respect to the large cardinals used in the classical proofs, namely Woodin cardinals---in our present proofs, coming from the supercompact cardinals, is put to work at forcing full choice. It is perhaps interesting to note that no reduction in the large cardinals needed seems to be available to our arguments if one replaces $\DC$ by some higher version $\DC_\lambda$ of it, for some fixed infinite cardinal $\lambda$, even if we restrict to forcing notions preserving $\DC_\lambda$, as we will still need to force $\AC$ in order to be able to apply {\L}o\'s's Theorem. It is only when we throw in the Axiom of Choice that we manage to drop our large cardinal assumption from many supercompact cardinals down to Woodin cardinals.  

On the other hand, and as we saw in \autoref{Motivation and limitations}, $\DC$ is provably needed if we are to derive generic absoluteness for the Chang model relative to any large cardinals whatsoever. The place where our argument breaks down if we remove $\DC$ from our hypotheses is that we will still be able to force $\ZFC$ over initial segments $V_\alpha$ of $V$, but we will definitely change the Chang model when doing so as we cannot assume, in the absence of $\DC$, that the forcing for getting $\AC$ is $\omega_1$-distributive. 

The above considerations show that it is reasonable to take $\ZF+\DC$, rather than $\ZFC$, as a vantage point in which to look at the generic absoluteness phenomenon for the Chang model in the presence of large cardinals. One could in fact assess the results in this second part of the paper as showing that $\DC$ is the bare minimum amount of choice enabling large cardinals to \textit{do their job} of fixing the theory of the Chang model. This is another argument in favour of the naturalness of $\ZF+\DC$: This theory not only provides the right framework for developing classical analysis, but is also the right base theory over which to derive generic absoluteness for $L(\mathbb R)$---where analysis lives---and even the Chang model, and thereby safeguarding truth in this inner model---and therefore truth in analysis---from the independence phenomenon.    

We will conclude by pointing out an obstacle for extending this type of results. The classical arguments for deriving generic absoluteness in $\ZFC$ for Chang models from large cardinals, relative to all set-forcings, work of course \textit{only} for the $\omega$-Chang model, which is precisely what we can capture---in the sense of leaving unchanged---by a $\omega_1$-distributive forcing assuming $\DC$ holds.\footnote{The $\omega_1$-Chang is correct about whether or not $\CH$ holds and, as mentioned in \autoref{Motivation and limitations}, both $\CH$ and $\lnot\CH$ can be forced while preserving all large cardinals.} On the other hand it is possible to prove reasonable generic absoluteness results for higher Chang models, provided the scope of the result is restricted to apply to forcing extension via a carefully chosen---but, arguably, still natural---class of partial orders. Specifically, for any given infinite regular cardinal $\kappa$ and any definable sufficiently nice class $\Gamma$ of forcing notions one can define a certain strong forcing axiom $\CFA(\Gamma)_\kappa$ which in particular implies the usual forcing axiom for meeting collections of $\kappa$-many dense subsets of forcings in $\Gamma$. One can then prove, under the assumption that sufficiently strong large cardinals exist and that  $\CFA(\Gamma)_\kappa$ holds, that the theory of the $\kappa$-Chang model is invariant with respect to forcing extensions via partial orders in $\Gamma$ which, in addition, preserve our forcing axiom $\CFA(\Gamma)_\kappa$.\footnote{Compare this to the restriction of our generic absoluteness result for the $\omega$-Chang model to forcing notions preserving $\DC$.} 
Moreover, for $\kappa=\omega_1$, many examples of nice classes $\Gamma$ of forcing notions have been found for which the corresponding forcing axiom $\CFA(\Gamma)_{\omega_1}$ is consistent (see \cite{Viale} and \cite{Aspero-Viale} for these results). It is totally unclear how one may want to try to adapt these results to a $\DC_\lambda$ context. The problem is that the formulation of the relevant strong forcing axiom completely loses its power in the absence of \L{}o\'{s}'s Theorem, to be applied, in the proofs, to certain ultrapowers coming from towers of ideals  belonging to our class $\Gamma$.\footnote{To be somewhat more specific, the forcing axioms of the form $\CFA(\Gamma)_\kappa$ assert that every forcing in $\Gamma$ can be completely embedded, with quotient in $\Gamma$, in a tower $\mathcal T$ of ideals on sets of size $\kappa$ such that $\mathcal T\in\Gamma$ and such that, moreover, $\mathcal T$ has certain nice structural properties within the class of forcings in $\Gamma$.} Finally, a (vague) question.

\begin{question} Can we find reasonable axioms $R$ for which one can prove, in $\ZF+\DC_\lambda$, for some $\lambda\geq\omega_1$, that if $R$ holds, then no forcing notion in some natural class which moreover preserves $R$ can change the theory of the $\omega_1$-Chang model?
\end{question}

\section*{Acknowledgements}
The authors would like to thank both Mirna D\v{z}amonja and Yasuo Yoshinobu for their helpful suggestions. Additional thanks are due to the organizers of the HIF programme for extending an invitation to Cambridge in October 2015, where the authors first met and started the work that culminated in this paper. Thanks are also due to the anonymous referee whose remarks helped elevate the quality of the paper in content and exposition. 
\bibliographystyle{amsplain}
\providecommand{\bysame}{\leavevmode\hbox to3em{\hrulefill}\thinspace}
\providecommand{\MR}{\relax\ifhmode\unskip\space\fi MR }
\providecommand{\MRhref}[2]{%
  \href{http://www.ams.org/mathscinet-getitem?mr=#1}{#2}
}
\providecommand{\href}[2]{#2}

\end{document}